\documentclass{amsart}
\usepackage{amsmath,amsthm}
\usepackage{amssymb}

\usepackage{graphicx}

\newtheorem{theorem}{Theorem}[section]
\newtheorem{proposition}[theorem]{Proposition}
\newtheorem{lemma}[theorem]{Lemma}
\newtheorem{corollary}[theorem]{Corollary}

\theoremstyle{definition}

\newtheorem{example}[theorem]{Example}

\def\img#1{\mathrm{Img}(#1)}

\begin{document}


\title{The Hadamard multiary quasigroup product}

    \author[R.M. Falc\'on]{Ra\'ul M. Falc\'on}
\address{Department of Applied Mathematics I, Universidad de Sevilla,
41012 Seville, Spain\\
ORCID: 0000-0002-6474-7301}\email{rafalgan@us.es}

\author[L. Mella]{Lorenzo Mella}
\address{Dip. di Scienze Fisiche, Informatiche, Matematiche, Universit\`a degli Studi di Modena e\\ Reggio Emilia, Via Campi 213/A, I-41125 Modena, Italy\\
ORCID: 0000-0002-6243-4897}\email{lorenzo.mella@unipr.it}

\author[P. Vojt\v{e}chovsk\'{y}]{Petr Vojt\v{e}chovsk\'{y}}
\address{Department of Mathematics, University of Denver, 2390 S. York St., Denver,
CO 80208, USA\\
ORCID: 0000-0003-3085-6611}
\email{petr.vojtechovsky@du.edu}

\begin{abstract}
The Hadamard quasigroup product has recently been introduced as a natural generalization of the classical Hadamard product of matrices. It is defined as the superposition operator of three binary operations, one of them being a quasigroup operation. This paper delves into the fundamentals of this superposition operator by considering its more general version over multiary groupoids. Particularly, we show how  this operator preserves algebraic identities, multiary groupoid structures, inverse
elements, isotopes, conjugates and orthogonality. Then, we generalize the mentioned Hadamard quasigroup product to  multiary quasigroups. Based on this product, we prove that the number of $m$-ary quasigroups defined on a given set $X$ coincides with the number of $m$-ary operations that are orthogonal to a given $m$-set of orthogonal $m$-ary operations over $X$.
\end{abstract}
\maketitle

\section{Introduction}\label{sec:introduction}

Let $\Omega_m(X)$ denote the set of $m$-ary operations on a set $X$, with $m\geq 2$. An $m$-ary {\em groupoid} over $X$ is any pair $(X,f)$, with $f\in\Omega_m(X)$. This pair is an {\em $m$-ary monoid} if it has an {\em identity element} $e\in X$. That is, 
\[f(\underbrace{e,\ldots,e}_{i},a,\underbrace{e,\ldots,e}_{m-i-1})=a\]
for all $a\in X$ and every non-negative integer $i<m$. Furthermore, the $m$-ary groupoid $(X,f)$ is an {\em $m$-ary quasigroup} if the following statement holds for each positive integer $i\leq m$: For every $(a_1,\ldots,a_{i-1},a_{i+1},...,a_m)\in X^{m-1}$ and every $b\in X$, there exists a unique $a_i\in X$ such that $f(a_1,\dots,a_{i-1},a_i,a_{i+1},\dots,a_m)=b$. If this is the case, then the Cayley table of $f$ is a {\em Latin hypercube} of dimension $m$. That is, an $m$-dimensional array with entries in $X$ such that, when $m-1$ coordinates are fixed, then every element in $X$ appears exactly once. This is a {\em Latin square} if $m=2$, and a {\em Latin cube} if $m=3$. A {\em Latin transversal} in $(X,f)$ is any set of $|X|$ cells in its Cayley table, no two of which coincide in any coordinate or contain the same symbol. From here on, we denote by $\mathcal{Q}_m(X)\subset\Omega_m(X)$ the subset of $m$-ary operations on $X$ describing an $m$-ary quasigroup.

\vspace{0.1cm}

An $n$-ary operation is called a {\em superposition operator} if it results after applying a finite number of times some of the following procedures: permute, identify or add variables in an operation, or replace variables by operations. (See \cite[Part II, Chapter 1]{Lau2006} for a more comprehensive development of this concept.) This paper focuses on the superposition operator that is defined as follows. Let $f\in\Omega_m(X)$ and $g_1,\ldots,g_m\in\Omega_n(X)$. Then, we define the superposition operator $\odot_f(g_1,\ldots,g_m)\in\Omega_n(X)$ such that 

\begin{equation}\label{eq_odot}
\odot_f(g_1,\ldots,g_m)\left(\overline{a}\right):=f\left(g_1\left(\overline{a}\right),\ldots,g_m\left(\overline{a}\right)\right)
\end{equation}
for all $\overline{a}:=(a_1,\ldots,a_n)\in X^n$. Moreover, we define the operator
\begin{equation}\label{eq_operator}\begin{array}{cccc}
\odot: & \Omega_m(X) & \to & \Omega_m(\Omega_n(X))\\
& f & \mapsto &\begin{array}{cccc}
 \odot_f: & \left(\Omega_n(X)\right)^m & \to & \Omega_n(X)\\
& (g_1,\ldots,g_m) & \mapsto & \odot_f(g_1,\ldots,g_m).
\end{array}
\end{array}
\end{equation}

It is said that $f$ {\em preserves} a subset $S\subseteq\Omega_n(X)$ if $f(g_1,\ldots,g_m)\in S$, for all $g_1,\ldots,g_m\in S$. This notion plays a relevant role in the theory of clones in universal algebra \cite{Butkoke2008,Butkoke2008a,Denecke2019,Lau2006}. Furthermore, the superposition operator (\ref{eq_odot}) allows the construction of $m$-ary operations endowing $\Omega_n(X)$ with a semigroup structure. (That is, it gives rise to an associative groupoid.) Thus, for instance, if $m=n$ and $g_1=\ldots=g_m=g$ in (\ref{eq_odot}), then $\Omega_n(X)$ is endowed with semigroup structure with the superposition operator $f+g:=\odot_f(g,\ldots,g)$ (see \cite{Butkoke2008}). Moreover, if $m=n=2$ and we have three binary operations $\star$, $*$ and $\circ$ on $X$, then the superposition operator (\ref{eq_odot}) is defined so that
\begin{equation}\label{eq_odot2}
\odot_\star(*,\circ)(a,b):= \left(a*b\right)\star \left(a\circ b\right)
\end{equation}
for all $a,b\in X$.  The following two binary operations arise from this operator.
\begin{itemize}
    \item The pair $\left(\Omega_2(X),\rhd\right)$ is a monoid \cite{Mezera2015, Przytycki2011}, where 
    \[a*\rhd\star b:=(a*b)\star b\]
    for all $*,\star\in\Omega_2(X)$ and $a,b\in X$. (Note that $*\rhd\star=\odot_\star(*,\circ)$, where $a\circ b=b$, for all $a,b\in X$.) Its identity element is the left projection operation $a*b = a$. Based on this operation, every group $(X,\cdot)$ can be embedded into $\Omega_2(X)$ so that the image of every $a\in X$ in $\Omega_2(X)$ is a right distributive operation on $X$ (see \cite{Mezera2015}).
    
    \item The pair $\left(\Omega_2(X),\lhd\right)$ is a non-commutative monoid \cite{Lopez2022}, where 
    \[a*\lhd\star b:=(a*b)\star (b*a)\]
    for all $*,\star\in\Omega_2(X)$ and $a,b\in X$. (Note that $*\lhd\star=\odot_\star(*,\circ)$, where  $a\circ b=b*a$, for all $a,b\in X$.) Its identity element is again the left projection operation.
\end{itemize}

Furthermore, for each binary operation $\star\in\Omega_2(X)$, let $\mathrm{Mult}_\star(X)$ be the subset of binary operations $*\in \Omega_2(X)$ that are distributive over $\star$. That is, 
\[a*(b\star c) = (a*b)\star (a*c) \hspace{0.5cm} \text{ and } \hspace{0.5cm}  (a\star b)*c = (a*c)\star(b*c)\]
for all $a,b,c\in X$. Fuchs \cite[Chapter 18] {Fuchs1958} realized that, if $(X,\star)$ is an abelian group, then the superposition operator (\ref{eq_odot2}) makes the pair $\left(\mathrm{Mult}_\star(X),\odot_\star\right)$ to be also an abelian group. Its identity element $*_e\in\mathrm{Mult}_\star(X)$ is defined so that $a*_e b:=e$ for all $a,b\in X$, where $e$ is the unit element of the group $(X,\star)$. In addition, the inverse of a binary operation $*\in\mathrm{Mult}_\star(X)$ is the binary operation $*^{-1}\in\mathrm{Mult}_\star(X)$ that is defined so that $a*^{-1} b:=(a*b)^{-1}$ for all $a,b\in X$. Fuchs also proved that the group $\left(\mathrm{Mult}_\star(X),\odot_\star\right)$ is isomorphic to the group of homomorphisms from $X\otimes X$ to $X$, and also to the group of homomorphisms from $X$ to the endomorphism group of $X$. Clay \cite{Clay1968} proved a similar result for the set ${\mathrm{Mult}_\star}_L(X)$ of binary operations $*\in\Omega_2(X)$ that are left distributive with respect to $\star$. More precisely, he proved that, if $(X,\star)$ is an abelian group, then the pair $\left({\mathrm{Mult}_\star}_L(X),\,\odot_\star\right)$ is an abelian group isomorphic to the set of functions from $X$ to the endomorphism group of $X$.

Much more recently, the superposition operator (\ref{eq_odot2}) has been used in \cite{Falcon2023} to introduce the so-called {\em Hadamard quasigroup product} as a natural generalization of the classical Hadamard product of matrices. Note in this regard that, if $X$ is a finite set of $n$ elements, then the Cayley table of every binary operation $*\in\Omega_2(X)$ constitutes an array $A_*=(A_*[x,y])$  in the set $\mathcal{A}_n(X)$ of $n\times n$ arrays with entries in $X$, where $A_*[a,b]:=a*b$ for all $a,b\in X$. The superposition operator (\ref{eq_odot2}) constitutes a Hadamard quasigroup product whenever $\star\in\mathcal{Q}_2(X)$. That is, whenever $A_\star$ is a Latin square. If this is the case, then the Hadamard quasigroup product $\odot_\star$ of two matrices $A_*,\,A_\circ\in\mathcal{A}_n(X)$ is defined so that
\begin{equation}\label{eq_odot_H}
\left(A_* \odot_\star A_\circ\right)[a,b]:=A_\star\left[A_*[a,b],\,A_\circ[a,b]\right]
\end{equation}
for all $a,b\in X$. This can be seen as the Hadamard product of $A_*$ and $A_\circ$ with respect to the operation $\star$, rather than with respect to the usual multiplication of numbers. If $*=\circ=\star$, then the successive iteration of this product is endowed with a cyclic behaviour that allows the definition of new isomorphism invariants of quasigroups. The set of binary operations $\star\in\mathcal{Q}_2(X)$ preserving the set $\mathcal{Q}_2(X)$ requires the existence
of successive localized Latin transversals in $(X,\star)$ (see \cite{Falcon2023}).

In this paper, we generalize the Hadamard quasigroup product to multiary quasigroups and study its fundamentals. The paper is organized as follows. Section \ref{sec:preliminaries} deals with some preliminary concepts on multiary groupoids that are used throughout the paper. Then, we study in Section \ref{sec:odot} the operator $\odot$ described in (\ref{eq_operator}). Particularly, we show how this operator preserves algebraic identities, multiary groupoid structures, inverse elements, isotopes, conjugates and orthogonality. Finally, the mentioned Hadamard quasigroup product is generalized in Section \ref{sec:Hadamard} to multiary quasigroups. Based on this product, we prove that the number of $m$-ary quasigroups defined on a given set $X$ coincides with the number of $m$-ary operations that are orthogonal to a given $m$-set of orthogonal $m$-ary operations over $X$.

\section{Preliminaries} \label{sec:preliminaries}

In this section, we show some preliminary concepts on multiary groupoids that are used throughout the paper. For more details about this topic, we refer the reader to \cite{Belousov1966}.

Let $(X,f)$ be an $m$-ary groupoid. The {\em image} of $f$ is defined as the set 
\[\img{f}:=\left\{f(\overline{a})\colon\,\overline{a}\in X^m\right\}.\]
Further, for each non-negative integer $i<m$, we denote by 
\[I_i(X,f):=\left\{e\in X\colon\, f(\underbrace{e,\ldots,e}_i,a,\underbrace{e,\ldots,e}_{m-i-1})=a, \text{ for all } a\in X\right\}\]
the set of all {\em $i$-identity elements} of $(X,f)$. (If $m=2$, then $I_0(X)$ (respectively, $I_1(X)$) is the set of all {\em left} (respectively, {\em right}) {\em identity elements} of $(X,f)$.) In addition, we denote by
\[I(X,f):=\bigcap_{i=0}^{m-1} I_i(X,f)\] the set of all {\em identity elements} of $(X,f)$. The pair $(X,f)$ is an $m$-ary monoid if and only if  $I(X,f)\neq\emptyset$.

\vspace{0.15cm}

Let $(X,f)$ be an $m$-ary monoid. An element $a\in X$ is said to be {\em invertible} in $(X,f)$ if there exist an element $a^{-1}\in X$ and an identity element $e\in I(X,f)$ such that
\begin{equation}\label{eq_inverse}
f(\underbrace{a,\ldots,a}_i,a^{-1},\underbrace{a,\ldots,a}_{m-i-1})=e
\end{equation}
for every non-negative integer $i<m$. The element $a^{-1}$ is said to be an {\em inverse} of $a$ in $(X,f)$. We denote by $\mathrm{Inv}(a,f)$ the set of inverses of $a$ in $(X,f)$. Then, the pair $(X,f)$ is said to {\em have unique inverses} if $|\mathrm{Inv}(a,f)|=1$, for all $a\in X$. If this is the case, then $|I(X,f)|=1$.

\begin{example}\label{example_invertible}
Let $X=\{1,2,3\}$. We consider the ternary groupoid $(X,f)$ described by the following arrays in $\mathcal{A}_3(X)$
\[\begin{array}{ccccc}
A_1\equiv\begin{array}{|c|c|c|} \hline
 1 & 2 & 3\\ \hline
 2 & 1 & 1\\ \hline
 3 & 1 & 1\\ \hline
\end{array} & \hspace{1cm} &
A_2\equiv
\begin{array}{|c|c|c|} \hline
 2 & 1 & 2\\ \hline
 1 & 2 & 3\\ \hline
 2 & 3 & 2\\ \hline
\end{array} & \hspace{1cm} &
A_3\equiv
\begin{array}{|c|c|c|} \hline
 3 & 3 & 1\\ \hline
 3 & 3 & 3\\ \hline
 1 & 3 & 1\\ \hline
\end{array}
\end{array}\]
so that $f(i,j,k):=A_k[i,j]$, for all $i,j,k\in X$. In particular, $I_0(X,f)=I_1(X,f)=I_2(X,f)=\{1,2\}$. Thus, $I(X,f)=\mathrm{Inv}(1,f)=\mathrm{Inv}(2,f)=\{1,2\}$. Moreover, $\mathrm{Inv}(3,f)=\{1\}$. Hence, $(X,f)$ does not have unique inverses. 

Now, let $(X,g)$ be the ternary groupoid described by the following arrays in $\mathcal{A}_3(X)$
\[\begin{array}{ccccc}
B_1\equiv\begin{array}{|c|c|c|} \hline
 2 & 2 & 2\\ \hline
 2 & 1 & 1\\ \hline
 3 & 1 & 1\\ \hline
\end{array} & \hspace{1cm} &
B_2\equiv
\begin{array}{|c|c|c|} \hline
 1 & 1 & 2\\ \hline
 1 & 2 & 3\\ \hline
 2 & 3 & 2\\ \hline
\end{array} & \hspace{1cm} &
B_3\equiv
\begin{array}{|c|c|c|} \hline
 3 & 3 & 1\\ \hline
 3 & 3 & 3\\ \hline
 1 & 3 & 2\\ \hline
\end{array}
\end{array}\]
so that $g(i,j,k):=B_k[i,j]$, for all $i,j,k\in X$. In particular, $I_0(X,g)=I_1(X,g)=I_2(X,g)=\{2\}$. Thus, $I(X,g)=\{2\}$. Moreover, $\mathrm{Inv}(i,g)=\{i\}$, for $i\in\{1,3\}$. Hence, $(X,g)$ has unique inverses. \hfill $\lhd$
\end{example}

\vspace{0.25cm}

From now on, we denote the symmetric group on a set $X$ by $\mathcal{S}_X$. (By $\mathcal{S}_n$ if $X=[n]$ for some positive integer $n$.) Then, $(X,f)$ is said to be {\em isotopic} to an $m$-groupoid $(X,g)$ if there exist $m+1$ permutations $\pi_1,\ldots,\pi_{m+1}\in \mathcal{S}_X$ such that
\[g\left(\pi_1(a_1),\ldots,\pi_m(a_m)\right)=\pi_{m+1}\left(f(a_1,\ldots,a_m)\right)\]
for all $a_1,\ldots,a_m\in X$. The tuple  $(\pi_1,\ldots,\pi_{m+1})$ is called an {\em isotopism} from $(X,f)$ to $(X,g)$. This is an {\em isomorphism} if $\pi_1=\ldots=\pi_{m+1}$, in which case the $m$-groupoids $(X,f)$  and $(X,g)$ are said to be {\em isomorphic}. Furthermore, $(X,f)$ is said to be {\em conjugate} to $(X,g)$ if there exists a permutation $\pi\in \mathcal{S}_{m+1}$ such that 
\[g(a_{\pi(1)},\ldots,a_{\pi(m)})=a_{\pi(m+1)}\Leftrightarrow f(a_1,\ldots,a_m)=a_{m+1}\]
for all $a_1,\ldots,a_{m+1}\in X$. If this is the case, then $(X,g)$ is said to be the {\em $\pi$-conjugate} of $(X,f)$ and it is denoted $g=f^{\pi}$. This $\pi$-conjugate always exists when $\pi(m+1)=m+1$. Moreover, if $(X,f)$ is an $m$-ary quasigroup, then all its conjugates exist.

Finally, a set of {\em orthogonal} $m$-ary operations is any subset $\{g_1,\ldots,g_m\}\subseteq\Omega_m(X)$ such that the map 
\begin{equation}\label{eq:map_permutation2}\begin{array}{cccc}
& X^m & \to & X^m\\
& \overline{a} &\mapsto & \left(g_1\left(\overline{a}\right),\ldots,g_m\left(\overline{a}\right)\right)
\end{array}
\end{equation}
is a bijection. 

\section{The operator $\odot$}\label{sec:odot}

From here on, we assume that $f\in\Omega_m(X)$. Moreover, for each $c\in X$, we define the constant map
\begin{equation}\label{eq_constant}\begin{array}{cccc}
\gamma_c: & X^n & \to & X\\
& \overline{a} & \mapsto & 
\gamma_c(\overline{a}):=c
\end{array}
\end{equation}

\vspace{0.15cm}

\begin{lemma}\label{lemma_embedding} The map $c\mapsto \gamma_c$ is a homomorphism that embeds $(X,f)$ into $\left(\Omega_n(X),\,\odot_f\right)$.
\end{lemma}

\begin{proof} Let $c_1,\ldots,c_m\in X$ and $\overline{a}\in X^n$. We have from  (\ref{eq_odot})  that
\[\gamma_{f\left(c_1,\ldots,c_m\right)}(\overline{a})=f\left(\gamma_{c_1}\left(\overline{a}\right),\ldots,\gamma_{c_m}\left(\overline{a}\right)\right) = \odot_f(\gamma_{c_1},\ldots,\gamma_{c_m})(\overline{a}).\]
Hence, $\gamma_{f\left(c_1,\ldots,c_m\right)}=\odot_f(\gamma_{c_1},\ldots,\gamma_{c_m})$.
\end{proof}

\vspace{0.25cm}

The $m$-ary groupoid $\left(\Omega_n(X),\,\odot_f\right)$ inherits indeed the algebraic structure of $(X,f)$. More precisely, the embedding described in Lemma \ref{lemma_embedding} enables us to prove that the operator $\odot$ preserves algebraic identities. In this regard, let $\varphi(x_1,\dots,x_t)$ and $\psi(x_1,\dots,x_t)$ be two algebraic terms based on $m$-products of universally quantified variables. In both of them, if $m>2$, then we identify the $m$ variables associated to each $m$-product within a pair of curly brackets and separated by commas. Thus, for instance, we could consider for $(m,t)=(3,9)$ the terms 
\[\varphi(x_1,\dots,x_9)=\left\{\{x_1,x_2,x_3\},\{x_4,x_5,x_6\},\{x_7,x_8,x_9\}\right\}\]
and
\[\psi(x_1,\dots,x_9)=\left\{\{x_1,x_2,x_3\},\{x_4,\{x_5,x_6,x_7\},x_8\},x_9\right\}.\]
For $m=2$, these curly brackets are not necessary and the product of variables is simply written as a juxtaposition. Thus, for instance, we can consider the term $x_1x_2$ for $t=2$, or the term $(x_1x_2)x_3$ for $t=3$.

Further, from here on, we write the algebraic identity $\varphi(x_1,\dots,x_t)=\psi(x_1,\dots,x_t)$ simply as $\varphi=\psi$. The $m$-ary groupoid $(X,f)$ (respectively, $\left(\Omega_n(X),\,\odot_f\right)$) is said to satisfy the algebraic identity $\varphi=\psi$ if all the products in the equation refer to the operation $f$ (respectively, $\odot_f$), and $\varphi(a_1 ,\dots,a_t) = \psi(a_1 ,\dots,a_t)$ for all $a_1,\ldots,a_t\in X$ (respectively, 
$\varphi(g_1,\ldots,g_m) = \psi(g_1 ,\ldots,g_m)$ for all $g_1,\dots,g_m\in\Omega_n(X)$). Consider as an example commutativity for $m=2$, and suppose that $f$ is a binary operation $\star$. Then, the $m$-ary groupoids $(X,\star)$ and $(\Omega_n(X),\odot_\star)$ are commutative if and only if they satisfy the algebraic identity $x_1x_2=x_2x_1$. Equivalently, $a\star b=b\star a$ for all $a,b\in X$, and $\odot_\star(g_1,g_2)=\odot_\star(g_2,g_1)$ for every pair of $n$-ary operations $g_1,g_2\in\Omega_n(X)$. This last condition means that
\[g_1(\overline{a})\star g_2(\overline{a}) = g_2(\overline{a})\star g_1(\overline{a})\]
for all $\overline{a}\in X^n$. The following result shows how  the operator $\odot$ preserves algebraic identities.

\begin{lemma} \label{lemma_identities} The $m$-ary groupoid $(X,f)$ satisfies the algebraic identity $\varphi=\psi$ if and only if the $m$-ary groupoid $\left(\Omega_n(X),\,\odot_f\right)$ satisfies it.
\end{lemma}

\begin{proof}
Let $g_1,\ldots,g_m\in \Omega_n(X)$ and $\overline{a}\in X^n$. If the $m$-ary groupoid $(X,f)$ satisfies the algebraic identity $\varphi=\psi$, then we have that 
\[\varphi\left(g_1\left(\overline{a}\right),\ldots,g_t\left(\overline{a}\right)\right)=\psi\left(g_1\left(\overline{a}\right),\ldots,g_t\left(\overline{a}\right)\right)\]
with respect to the product $f$. From (\ref{eq_odot}), we have that $\varphi\left(g_1,\ldots,g_t\right)=\psi\left(g_1,\ldots,g_t\right)$ with respect to the product $\odot_f$. Hence, the $m$-ary groupoid $\left(\Omega_n(X),\,\odot_f\right)$ also satisfies the algebraic identity $\varphi=\psi$.

Conversely, let $a_1,\dots,a_t\in X$. If the $m$-ary groupoid $\left(\Omega_n(X),\,\odot_f\right)$ satisfies the algebraic identity $\varphi=\psi$, then we have that 
\[\varphi\left(\gamma_{a_1},\ldots,\gamma_{a_t}\right) = \psi\left(\gamma_{a_1},\ldots,\gamma_{a_t}\right)\]
with respect to the product $\odot_f$. Thus, from (\ref{eq_odot}) and (\ref{eq_constant}), we have that $\varphi(a_1,\dots,a_t)=\psi(a_1,\dots,a_t)$. Hence, the $m$-ary groupoid $(X,f)$ also satisfies the algebraic identity  $\varphi=\psi$.
\end{proof}

\vspace{0.25cm}

The operator $\odot$ also preserves $m$-ary monoid structures. In this regard, the following lemma describes all the $i$-identity elements in the $m$-ary groupoid $\left(\Omega_n(X),\,\odot_f\right)$, for every non-negative integer $i<m$.

\begin{lemma}\label{Lm:LI} The following statements hold.
\begin{enumerate}
\item Let $c\in X$. Then, $c\in I_i(X,f)$ if and only if $\gamma_c\in I_i(\Omega_n(X),\odot_f)$.
\item For each non-negative integer $i<m$,
\[I_i(\Omega_n(X),\odot_f) = \{g\in\Omega_n(X):\img{g}\subseteq I_i(X,f)\}.\]
\end{enumerate}
\end{lemma}

\begin{proof}
We prove separately each statement.
\begin{enumerate}
    \item If $c\in I_i(X,f)$, then
\begin{equation}\label{eq_LI_a}
f(\underbrace{c,\ldots,c}_i,a,\underbrace{c,\ldots,c}_{m-i-1})=a   
\end{equation} 
for all $a\in X$. As a consequence,
\[\odot_f(\underbrace{\gamma_c,\ldots,\gamma_c}_i,g,\underbrace{\gamma_c,\ldots,\gamma_c}_{m-i-1})(\overline{a})=g(\overline{a})\]
for all $g\in\Omega_n(X)$ and $\overline{a}\in X^n$. Hence, $\gamma_c\in I_i(\Omega_n(X),\odot_f)$.

\vspace{0.15cm}

Conversely, if $\gamma_c\in I_i(\Omega_n(X),\odot_f)$, then
\[\odot_f(\underbrace{\gamma_c,\ldots,\gamma_c}_i,g,\underbrace{\gamma_c,\ldots,\gamma_c}_{m-i-1})=g\]
for all $g\in\Omega_n(X)$. In particular, if $g=\gamma_a$, for some $a\in X$, then (\ref{eq_LI_a}) holds readily. Hence, $c\in I_i(X,f)$.

\vspace{0.25cm}

\item     
Let $g\in I_i(\Omega_n(X),\odot_f)$. It means that $g\in\Omega_n(X)$ satisfies that
\[\odot_f(\underbrace{g,\ldots,g}_i,h,\underbrace{g,\ldots,g}_{m-i-1})=h\]
for all $h\in\Omega_n(X)$. Particularly, if $h=\gamma_c$ for some $c\in X$, then we have that 
\begin{equation}\label{Eq:LI}
f(\underbrace{g(\overline{a}),\ldots,g(\overline{a})}_i,c,\underbrace{g(\overline{a}),\ldots,g(\overline{a})}_{m-i-1})=c
\end{equation}
for all $\overline{a}\in X^n$. Thus, $\img{g}\subseteq I_i(X,f)$.

\vspace{0.15cm}

Conversely, let $g\in\Omega_n(X)$ be such that $\img{g}\subseteq I_i(X,f)$. Then, (\ref{Eq:LI}) holds for all $c\in X$ and hence, 
\[\odot_f(\underbrace{g,\ldots,g}_i,h,\underbrace{g,\ldots,g}_{m-i-1})(\overline{a})=h(\overline{a})\]
for all $h\in \Omega_n(X)$ and $\overline{a}\in X^n$. Thus, $g\in I_i(\Omega_n(X),\odot_f)$.
\end{enumerate}
\hfill \,
\end{proof}

\vspace{0.25cm}

\begin{example}
Let $(X,\star)$ be defined by
\begin{displaymath}
    \begin{array}{c|ccc}
        \star&1&2&3\\
        \hline
        1&1&2&3\\
        2&1&2&3\\
        3&1&1&1
    \end{array}.
\end{displaymath}
Since $I_0(X,\star) = \{1,2\}$, Lemma \ref{Lm:LI} implies, for instance, that the binary operation
\begin{displaymath}
    \begin{array}{c|ccc}
        \circ&1&2&3\\
        \hline
        1&1&2&1\\
        2&2&2&1\\
        3&1&2&2
    \end{array}
\end{displaymath}
is a left identity element in $(\Omega_2(X),\odot_\star)$.\hfill $\lhd$
\end{example}

\vspace{0.25cm}

\begin{proposition}\label{proposition:LI}
It is satisfied that
\[I(\Omega_n(X),\odot_f) = \{g\in \Omega_n(X): \img{g}\subseteq I(X,f)\}.\]
In particular, the following two statements hold.
\begin{enumerate}
    \item If $I(X,f)=\{e\}$ for some $e\in X$, then $I(\Omega_n(X),\odot_f)=\{\gamma_e\}$.
    \item If $I(\Omega_n(X),\odot_f)=\{g\}$ for some $g\in\Omega_n(X)$, then $g=\gamma_e$ for some $e\in X$. Moreover, $I(X,f)=\{e\}$.
\end{enumerate}
\end{proposition}

\begin{proof} The result holds readily from Lemma \ref{Lm:LI}. Concerning the second statement, if $g\in I(\Omega_n(X),\odot_f)$ and $\img{g}$ is not a singleton, then $I(X,f)$ is not a singleton and there is therefore more than one element in $I(\Omega_n(X),\odot_f)$.    
\end{proof}

\vspace{0.05cm}

The operator $\odot$ also preserves inverse elements.

\begin{proposition}\label{proposition_inverse}
Let $(X,f)$ be an $m$-ary monoid. The following statements hold.
\begin{enumerate}
\item Let $g\in\Omega_n(X)$ be an invertible element in $(\Omega_n(X),\odot_f)$. Then,
\[\mathrm{Inv}\left(g,\,\odot_f\right)=\left\{h\in\Omega_n(X)\colon\, h(\overline{a})\in\mathrm{Inv}(g(\overline{a}),\,f) \text{ for all } \overline{a}\in X^n\right\}.\]

\item Let $c\in X$. If $c^{-1}\in \mathrm{Inv}(c,f)$, then $\gamma_{c^{-1}}\in\mathrm{Inv}\left(\gamma_c,\, \odot_f\right)$.

\item Let $c\in X$. If $(\gamma_c)^{-1}\in\mathrm{Inv}\left(\gamma_c,\, \odot_f\right)$, then $\img{(\gamma_c)^{-1}}\subseteq \mathrm{Inv}(c,f)$.

    \item If $(X,f)$ has unique inverses, then $(\Omega_n(X),\odot_f)$ also has unique inverses. More precisely, the inverse $g^{-1}$ of $g\in\Omega_n(X)$ is given by $g^{-1}(\overline{a}) = \left(g(\overline{a})\right)^{-1}$ for all $\overline{a}\in X^n$.
    
    \item If $(\Omega_n(X),\odot_f)$ has unique inverses, then $(X,f)$ also has unique inverses.
\end{enumerate}
\end{proposition}

\begin{proof} Proposition \ref{proposition:LI} implies that an $n$-ary operation $h\in\Omega_n(X)$ is an inverse of an invertible $g\in\Omega_n(X)$ in $\left(\Omega_n(X),\,\odot_f\right)$ if and only if, for each non-negative integer $i<m$ and each $\overline{a}\in X$, there exists an identity element $e_{i,\overline{a}}\in I(X,f)$ such that
\[\odot_f(\underbrace{g,\ldots,g}_i,h,\underbrace{g,\ldots,g}_{m-i-1})(\overline{a})=f(\underbrace{g(\overline{a}),\ldots,g(\overline{a})}_i,h(\overline{a}),\underbrace{g(\overline{a}),\ldots,g(\overline{a})}_{m-i-1})=e_{i,\overline{a}}.\]
That is, if and only if $h(\overline{a})\in\mathrm{Inv}(g(\overline{a}),\,f)$ for all $\overline{a}\in X^n$. Hence, the first statement holds.

Statements (2)--(4) follow readily from the first one. (Note in the third statement that, if $I(X,f)=\{e\}$, then Proposition \ref{proposition:LI} implies that $I(\Omega_n(X),\odot_f)=\{\gamma_e\}$.) Now, if $(\Omega_n(X),\odot_f)$ has unique inverses, then we have from Proposition \ref{proposition:LI} that $I(\Omega_n(X),\odot_f)=\{\gamma_e\}$ where  $I(X,f)=\{e\}$. In addition, we have from the first statement that, for each $c\in X$,  $\img{(\gamma_c)^{-1}}\subseteq \mathrm{Inv}(c,f)$. As a consequence, the first statement also implies that $\gamma_u\in\mathrm{Inv}\left(\gamma_c,\,\odot_f\right)$ for all $u\in\img{(\gamma_c)^{-1}}$. Since $(\Omega_n(X),\odot_f)$ has unique inverses, it must be $|\img{(\gamma_c)^{-1}}|=1$. Let $c^{-1}$ be the unique element in $\img{(\gamma_c)^{-1}}$. Then, $c^{-1}\in\mathrm{Inv}(c,f)$ from the third statement. It is necessarily unique from the second statement. Hence, the fifth statement holds.
\hfill \, 
\end{proof}

\vspace{0.25cm}

Now, we show how the operator $\odot$ preserves isotopisms and conjugates.

\begin{lemma}\label{lemma_isotopism} If $(X,f)$ is isotopic to an $m$-ary groupoid 
$(X,g)$, then $(\Omega_n(X),\odot_f)$ is isotopic to $(\Omega_n(X),\odot_g)$.
\end{lemma}

    \begin{proof}
    Let $(\pi_1,\ldots,\pi_{m+1})$ be an isotopism from $(X,f)$ to $(X,g)$. For each positive integer $i\leq m+1$, we define the map 

\[\begin{array}{cccc}
\rho_i:& \Omega_n(X) & \to & \Omega_n(X)\\
& h &\mapsto & \pi_i\circ h
\end{array}\]
where $\circ$ denotes the composition of maps. (That is, $\rho_i(h)(\overline{a}):=\pi_i(h(\overline{a}))$, for all $\overline{a}\in X^n$.) This map is a bijection because $\pi_i\in\mathcal{S}_X$. Finally, let $h_1,\ldots,h_m\in\Omega_n(X)$ and $\overline{a}\in X^n$. Since $(\pi_1,\ldots,\pi_{m+1})$ is an isotopism from $(X,f)$ to $(X,g)$, we have that
\begin{align*}
\odot_g\left(\rho_1(h_1),\ldots,\rho_m(h_m)\right)(\overline{a}) & =g(\pi_1(h_1(\overline{a})),\ldots,\pi_m(h_m(\overline{a})))=\\
& =\pi_{m+1}(f(h_1(\overline{a}),\ldots,h_m(\overline{a})))=\\
& =\pi_{m+1}(\odot_f(h_1,\ldots,h_m)(\overline{a}))=\\
& =\rho_{m+1}(\odot_f(h_1,\ldots,h_m))(\overline{a}).
\end{align*}
Thus, $\odot_g\left(\rho_1(h_1),\ldots,\rho_m(h_m)\right)=\rho_{m+1}(\odot_f(h_1,\ldots,h_m))$ and hence, $(\rho_1,\ldots,\rho_{m+1})$ is an isotopism from $(\Omega_n(X),\odot_f)$ to $(\Omega_n(X),\odot_g)$.
\end{proof}

\vspace{0.25cm}

\begin{lemma}\label{lemma_conjugate} $(X,f)$ is conjugate to an $m$-ary groupoid 
$(X,g)$ if and only if $(\Omega_n(X),\odot_f)$ is conjugate to $(\Omega_n(X),\odot_g)$. Moreover, if $g=f^\pi$ for some $\pi\in \mathcal{S}_{m+1}$, then $\odot_g=(\odot_f)^\pi$, and reciprocally.
\end{lemma}

\begin{proof}
Let $\pi \in \mathcal{S}_{m+1}$ be such that $g=f^\pi$. For each $h_1,\ldots,h_{m+1}\in \Omega_n(X)$ and $\overline{a}\in X^n$, we have that

\begin{align*}
\odot_f(h_1,\ldots,h_m)(\overline{a})=h_{m+1}(\overline{a}) & \Leftrightarrow f\left(h_1(\overline{a}),\ldots,h_m(\overline{a})\right)=h_{m+1}(\overline{a}) \Leftrightarrow \\
& \Leftrightarrow f^\pi\left(h_{\pi(1)}(\overline{a}),\ldots,h_{\pi(m)}(\overline{a})\right)=h_{\pi(m+1)}(\overline{a}) \Leftrightarrow \\
& \Leftrightarrow \odot_{f^\pi}\left(h_{\pi(1)},\ldots,h_{\pi(m)}\right)(\overline{a})=h_{\pi(m+1)}(\overline{a}).
\end{align*}
Hence, $(\Omega_n(X),\odot_f)$ and $(\Omega_n(X),\odot_g)$ are conjugate and $\odot_g=(\odot_f)^\pi$.

Conversely, let $\pi \in \mathcal{S}_{m+1}$ be such that $\odot_g=(\odot_f)^\pi$. For each $c_1,\ldots,c_{m+1}\in X$, we have that
\begin{align*}
f(c_1,\ldots,c_m)=c_{m+1} & \Leftrightarrow \odot_f\left(\gamma_{c_1},\ldots,\gamma_{c_m}\right)=\gamma_{c_{m+1}} \Leftrightarrow \\
& \Leftrightarrow (\odot_f)^\pi\left(\gamma_{c_{\pi(1)}},\ldots,\gamma_{c_{\pi(m)}}\right)=\gamma_{c_{\pi(m+1)}} \Leftrightarrow \\
& \Leftrightarrow g\left(c_{\pi(1)},\ldots,c_{\pi(m)}\right)=c_{\pi(m+1)}.
\end{align*}
Hence, $(X,f)$ and $(X,g)$ are conjugate and $g=f^\pi$.
\end{proof}

\vspace{0.25cm}

\begin{lemma}\label{lemma_conjugate_2} Let $\pi\in \mathcal{S}_{n+1}$ be such that $\pi(n+1)=n+1$. Then, for each $g_1,\ldots,g_m\in\Omega_n(X)$, 
\[\odot_f(g_1^\pi,\ldots,g_m^\pi)=\left(\odot_f(g_1,\ldots,g_m)\right)^\pi.\]
\end{lemma}

\begin{proof} For each $\overline{a}:=(a_1,\ldots,a_n)\in X^n$, let $\overline{a}^{\pi}:=(a_{\pi(1)},\ldots,a_{\pi(n)})$. Since $\pi(n+1)=n+1$, we have that

\begin{align*}
\odot_f(g_1^\pi,\ldots,g_m^\pi)(\overline{a}^\pi) & =f(g_1^\pi(\overline{a}^\pi),\ldots,g_m^\pi(\overline{a}^\pi))=\\
&=f(g_1(\overline{a}),\ldots,g_m(\overline{a}))=\\
& =\odot_f(g_1,\ldots,g_m)(\overline{a})=\\
& =(\odot_f(g_1,\ldots,g_m))^\pi(\overline{a}^\pi).
\end{align*}
Hence, the result holds.
\end{proof}

\vspace{0.25cm}

We finish this section with a preliminary result concerning orthogonality. Thus, we focus on the case $m=n$.

\begin{lemma}\label{lemma_permutation} Let $\{g_1,\ldots,g_m\}\subseteq \Omega_m(X)$. This is a set of orthogonal $m$-ary operations if and only if the map
\begin{equation}\label{eq:map_permutation}
\begin{array}{cccc}
& \Omega_m(X) & \to & \Omega_m(X)\\
& f &\mapsto & \odot_f(g_1,\ldots,g_m)
\end{array}
\end{equation}
is a bijection.
\end{lemma}

\begin{proof} Let $\{g_1,\ldots,g_m\}\subseteq \Omega_m(X)$ be a set of orthogonal $m$-ary operations. First, we prove that the map (\ref{eq:map_permutation}) is onto. For each $h\in\Omega_m(X)$, let $f_h\in \Omega_m(X)$ be defined so that $f_h(g_1(\overline{a}),\ldots,g_m(\overline{a}))=h(\overline{a})$, for all $\overline{a}\in X^m$. It is well-defined because $\{g_1,\ldots,g_m\}$ is a set of orthogonal $m$-ary operations and thus, the map  (\ref{eq:map_permutation2}) is a bijection. As a consequence, $\odot_{f_h}(g_1,\ldots,g_m)=h$ and hence, the map (\ref{eq:map_permutation}) is onto.

In order to prove that the map (\ref{eq:map_permutation}) is injective, let $f_1,f_2\in\Omega_m(X)$ be such that $\odot_{f_1}(g_1,\ldots,g_m)=\odot_{f_2}(g_1,\ldots,g_m)$. Then, 
\begin{equation}\label{eq:lemma_permutation}
f_1(g_1(\overline{a}),\ldots,g_m(\overline{a}))=f_2(g_1(\overline{a}),\ldots,g_m(\overline{a}))
\end{equation}
for all $\overline{a}\in X^m$. If $\{g_1,\ldots,g_m\}$ is a set of orthogonal $m$-ary operations, then (\ref{eq:lemma_permutation}) implies that $f_1(\overline{a})=f_2(\overline{a})$ for all $\overline{a}\in X^m$. So, $f_1=f_2$ and hence, the map (\ref{eq:map_permutation}) is injective.

For the converse, suppose that the map (\ref{eq:map_permutation}) is a bijection, but the map (\ref{eq:map_permutation2}) is not a bijection. If this last map is not onto, then there is an element $\overline{a}\in X^m$ such that $(g_1(\overline{b}),\ldots,g_m(\overline{b}))\neq \overline{a}$, for all $\overline{b}\in X^m$. If $|X|=1$, then $|\Omega_m(X)|=1$ and $\{g_1,\dots,g_m\}$ is a set of orthogonal $m$-ary operations mapping the unique $m$-tuple in $X^m$ to itself. So, we may assume that $|X|>1$. Let $x,y\in X$ be such that $x\neq y$. Then, let $f_1,f_2\in \Omega_m(X)$ be defined so that 
\begin{itemize}
    \item $f_1\left(g_1(\overline{b}),\ldots,g_m(\overline{b})\right)=f_2\left(g_1(\overline{b}),\ldots,g_m(\overline{b})\right)=x$, for all $\overline{b}\in X^m$;
    \item $f_1(\overline{a})=x$;
    \item $f_2(\overline{a})=y$; and
    \item the rest of $f_1$ and $f_2$ are arbitrarily defined.
\end{itemize}
Then, $\odot_{f_1}(g_1,\ldots,g_m)= \odot_{f_2}(g_1,\ldots,g_m)$, but $f_1\neq f_2$, so the map (\ref{eq:map_permutation}) is not injective, which is a contradiction.

Finally, suppose that the map (\ref{eq:map_permutation2}) is not one-to-one. Let $\overline{a}\ne\overline{b}\in X^m$ be such that $(g_1(\overline a),\dots,g_m(\overline a)) = (g_1(\overline b),\dots,g_m(\overline b))$. Then, for every $f\in\Omega_m(X)$, we have that $\odot_f(g_1,\dots,g_m)(\overline a) = \odot_f(g_1,\dots,g_m)(\overline b)$. But then, any $h\in\Omega_m(X)$ such that $h(\overline a)\ne h(\overline b)$ shows that the map (\ref{eq:map_permutation}) is not onto, which is a contradicition.
\end{proof}

\section{The Hadamard multiary quasigroup product}\label{sec:Hadamard}

In what follows, we focus on the case in which $(X,f)$ is an $m$-ary quasigroup. That is, $f\in\mathcal{Q}_m(X)$. If $m=2$, then the operation $\odot_f$ is the Hadamard quasigroup product described in (\ref{eq_odot_H}). First, we show how the quasigroup structure of $(X,f)$ is preserved by the operator $\odot$. To this end, we make use of the embedding described in Lemma \ref{lemma_embedding}.

\begin{proposition}\label{proposition_quasigroup}
$f\in\mathcal{Q}_m(X)$ if and only if $\odot_f\in\mathcal{Q}_m\left(\Omega_n(X)\right)$.
\end{proposition}

\begin{proof}
Let $f\in\mathcal{Q}_m(X)$ and $g_1,\ldots,g_m\in\Omega_n(X)$. Then, let $h:=\odot_f(g_1,\ldots,g_m)\in\Omega_n$ and  $\overline{x}\in X^n$. From (\ref{eq_odot}), we have that
\begin{equation}\label{Eq:Q}
f\left(g_1\left(\overline{x}\right),\ldots,g_m\left(\overline{x}\right)\right)=h\left(\overline{x}\right).
\end{equation}
Since $f\in\mathcal{Q}_m(X)$, for each $i\in\{1,\ldots,m\}$, there exists a unique  $a_{\overline{x},i}\in X$ such that 
\[f\left(g_1\left(\overline{x}\right),\ldots,g_{i-1}\left(\overline{x}\right),\,a_{\overline{x},i},\, g_{i+1}\left(\overline{x}\right),\,\ldots,g_m\left(\overline{x}\right)\right)=h\left(\overline{x}\right).\]
Thus, given the set $\{g_1,\ldots,g_{i-1},\,g_{i+1},\ldots,g_m,\,h\}$, there is a unique $g_i\in\Omega_m(X)$ satisfying \eqref{Eq:Q}, namely $g_i(\overline{x}):=a_{\overline{x},i}$. Hence, $\odot_f\in\mathcal{Q}_m\left(\Omega_n(X)\right)$.

Conversely, suppose that $\odot_f\in\mathcal{Q}_m\left(\Omega_n(X)\right)$. Let $a_1,\ldots,a_m\in X$  and $i\in\{1,\ldots,m\}$ be such that there exists $a'_i\in X$ satisfying that  
\[f(a_1,\ldots,a_{i-1},a_i,a_{i+1},\ldots,a_m)=f(a_1,\ldots,a_{i-1},a'_i,a_{i+1},\ldots,a_m).\]
From Lemma \ref{lemma_embedding}, we have that
\begin{align*}
\odot_f(\gamma_{a_1},\ldots,\gamma_{a_{i-1}},\,\gamma_{a_i},\,\gamma_{a_{i+1}},\,\ldots,\gamma_{a_m}) & = \gamma_{f(a_1,\ldots,a_{i-1},a_i,a_{i+1},\ldots,a_m)}=\\
& =\gamma_{f(a_1,\ldots,a_{i-1},a'_i,a_{i+1},\ldots,a_m)}=\\
& = \odot_f(\gamma_{a_1},\ldots,\gamma_{a_{i-1}},\,\gamma_{a'_i},\,\gamma_{a_{i+1}},\,\ldots,\gamma_{a_m}).   
\end{align*}
Since $\odot_f\in\mathcal{Q}_m\left(\Omega_n(X)\right)$, it must be $\gamma_{a_i}=\gamma_{a'_i}$ and thus, $a_i=a'_i$. Hence, $f\in\mathcal{Q}_m(X)$.
\end{proof}

\vspace{0.15cm}

Next, we show a necessary and sufficient condition for ensuring
that the product $\odot_f$ of two $m$-ary quasigroups is also an $m$-ary quasigroup. This condition generalizes a previous one described in \cite[Lemma 12]{Falcon2023} for the bidimensional case ($m=2$).

\begin{lemma}\label{lemma_quasigroups} Let $f\in\mathcal{Q}_m(X)$ and $g_1,\ldots,g_m\in\mathcal{Q}_n(X)$. Then, $\odot_f(g_1,\ldots,g_m)\in \mathcal{Q}_n(X)$ if and only if, for each $(a_1,\ldots,a_{n-1})\in X^{n-1}$ and each positive integer $i\leq n$, the set
{\small \[\left\{\left(g_1(a_1,\ldots,a_{i-1},x,a_i,a_{i+1},\ldots,a_{n-1}),\ldots,g_m(a_1,\ldots,a_{i-1},x,a_i,a_{i+1},\ldots,a_{n-1})\right)\colon\,x\in X\right\}\]}
is a Latin transversal in $(X,f)$.
\end{lemma}

\begin{proof} Since $g_1,\ldots,g_m\in\mathcal{Q}_n(X)$, no two cells in the set of the statement coincide in any coordinate. Then, the result holds because $\odot_f(g_1,\ldots,g_m)\in \mathcal{Q}_n(X)$ if and only if
\[\left\{\odot_f\left(g_1,\ldots,g_m\right)(a_1,\ldots,a_{i-1},x,a_i,a_{i+1},\ldots,a_{n-1})\colon\,x\in X\right\}=X.\]
\hfill \,
\end{proof}

\vspace{0.15cm}

The quasigroup structure of $(X,f)$ makes the operator $\odot$ to preserve even more properties than those ones shown in the previous section. We delve into this aspect by focusing on the orthogonality condition. Thus, we assume from now on that $m=n$. For each subset $S=\{g_1,\ldots,g_m\}\subseteq \Omega_m(X)$, we define the set
{\small \begin{equation}\label{eq:ort}
\mathrm{Ort}(S):=\left\{g\in\Omega_m(X)\colon \left(S\setminus\{g_i\}\right) \cup\{g\} \text{ is a set of orthogonal operations } \forall i\leq m\right\}.
\end{equation}}

The next result particularizes Lemma \ref{lemma_permutation} in case of dealing with $m$-ary quasigroups on the set $X$ instead of the whole set $\Omega_m(X)$. It shows in particular that the number of $m$-ary quasigroups on $X$ coincides with the cardinality of $\mathrm{Ort}(S)$ for any set of orthogonal $m$-ary operations over $X$.

\begin{theorem} \label{theorem_orthogonalQ} Let $S=\{g_1,\ldots,g_m\}\subseteq \Omega_m(X)$ be a set of orthogonal $m$-ary operations. Then, the map
\begin{equation}\label{eq:map_ort}
\begin{array}{cccc}
& \mathcal{Q}_m(X) & \rightarrow & \mathrm{Ort}(S)\\
& f &\rightarrow& \odot_f(g_1,\ldots,g_m)
\end{array}
\end{equation}
is a bijection.
\end{theorem}

\begin{proof}
The map (\ref{eq:map_ort}) is well-defined. To see it, assume that $f\in\Omega_m(X)$ and let us prove that  $\odot_f(g_1,\ldots,g_m)\in \mathrm{Ort}(S)$. To this end, we shall see that $\{\odot_f(g_1,\ldots,g_m),$ $g_2,\ldots,g_m\}$ is a set of orthogonal $m$-ary products. (The remaining cases in the definition (\ref{eq:ort}) follow similarly.) So, we have to prove that the map
\begin{equation}\label{eq:map_orthogonalQ}
\begin{array}{cccc}
& X^m & \to & X^m\\
& \overline{a} &\mapsto & \left(\odot_f(g_1,\ldots,g_m)(\overline{a}),g_2(\overline{a}),\ldots,g_m(\overline{a})\right)
\end{array}
\end{equation}
is a bijection. First, we prove that it is injective. To this end, let $\overline{a},\overline{b}\in X^m$ be such that
\[\left(\odot_f(g_1,\ldots,g_m)(\overline{a}),g_2(\overline{a}),\ldots,g_m(\overline{a})\right)=\left(\odot_f(g_1,\ldots,g_m)(\overline{b}),g_2(\overline{b}),\ldots,g_m(\overline{b})\right).\]
Then, $g_i(\overline{a})=g_i(\overline{b})$ for all $i\in\{2,\ldots,m\}$, and
\[f\left(g_1(\overline{a}),\ldots,g_m(\overline{a})\right)=f\left(g_1(\overline{b}),\ldots,g_m(\overline{b})\right).\]
Since $f\in\mathcal{Q}_m(X)$, it must be $g_1(\overline{a})=g_1(\overline{b})$. But then, $\overline{a}=\overline{b}$ because $S$ is a set of orthogonal $m$-ary operations. Hence, the map (\ref{eq:map_orthogonalQ}) is injective. 

Second, we prove that the map (\ref{eq:map_orthogonalQ}) is onto. To this end, let $\overline{a}:=(a_1,\ldots,a_m)\in X^m$. Since $f\in\mathcal{Q}_m(X)$, there exists an element $x\in X$ such that $f(x,a_2,\ldots,a_m)=a_1$. Moreover, since $S$ is a set of orthogonal $m$-ary operations, there exists an element $\overline{b}\in X^m$ such that $g_1(\overline{b})=x$ and $g_i(\overline{b})=a_i$ for all $i\in\{2,\ldots,m\}$. Hence, $\left(\odot_f(g_1,\ldots,g_m)(\overline{b}),g_2(\overline{b}),\ldots,g_m(\overline{b})\right)=\overline{a}$. Thus, the map (\ref{eq:map_orthogonalQ}) is onto and hence, the map (\ref{eq:map_ort}) is well-defined.

Further, the injectivity of the map (\ref{eq:map_ort}) follows from Lemma \ref{lemma_permutation}. It therefore remains to prove that this map is also onto. To this end, let $h\in \mathrm{Ort}(S)$. Then, let $f\in \Omega_m(X)$ be defined so that
\[\odot_f(g_1,\ldots,g_m)(\overline{a})=h(\overline{a})\]
for all $\overline{a}\in X^m$. (We have already seen in the proof of Lemma \ref{lemma_permutation} that this $m$-ary operation is well-defined.)  It remains to show that $f\in\mathcal{Q}_m(X)$. That is, we have to prove that, for every $(a_1,\ldots,a_{i-1},a_{i+1},...,a_m)\in X^{m-1}$ and every $b\in X$, there exists a unique $a_i\in X$ such that $f(a_1,\dots,a_{i-1},a_i,a_{i+1},\dots,a_m)=b$. Without loss of generality, we prove this condition for $i=1$. In this regard, let $(a_2,\ldots,a_m)\in X^{m-1}$ and $b\in X$. Since $\{h,g_2,\ldots,g_m\}$ is a set of orthogonal $m$-ary operations, there exists a unique element $\overline{c}\in X^m$ such that
\[\left(h(\overline{c}),g_2(\overline{c}),\ldots,g_m(\overline{c}))\right)=\left(b,a_2,\ldots,a_m\right).\]
Then, the element $g_1(\overline{c})\in X$ satisfies that
\[f(g_1(\overline{c}),a_2,\ldots,a_m)=\odot_f(g_1,\ldots,g_m)(\overline{c})=h(\overline{c})=b.\]
In order to prove that $f\in\mathcal{Q}_m(X)$, we have to prove that $g_1(\overline{c})$ is the unique element in $X$ satisfying the previous condition. To this end, let 
\[\overline{a}:=(a_1,a_2,\ldots,a_m)\in X^m\]
and 
\[\overline{a}':=(a'_1,a_2\ldots,a_m)\in X^m,\]
be such that $f(\overline{a})=f(\overline{a}')=b$. Since $\{g_1,\ldots,g_m\}$ is a set of orthogonal $m$-ary operations, there exist two elements $\overline{b},\overline{c}\in X^m$ such that $(g_1(\overline{b}),\ldots,g_m(\overline{b}))=\overline{a}$ and $(g_1(\overline{c}),\ldots,g_m(\overline{c}))=\overline{a}'$. Then,
\[h(\overline{b})=\odot_f(g_1,\ldots,g_m)(\overline{b})=f(\overline{a})=f(\overline{a}')=\odot_f(g_1,\ldots,g_m)(\overline{c})=h(\overline{c}).\]
As a consequence,
{\small \[ (h(\overline{b}),g_2(\overline{b}),\ldots,g_m(\overline{b}))=(h(\overline{c}),g_2(\overline{c}),\ldots,g_m(\overline{c})).\]}
Since $h\in\mathrm{Ort}(S)$, we have that $\overline{b}=\overline{c}$ and hence, $\overline{a}=\overline{a}'$. Therefore, $f\in\mathcal{Q}_m(X)$.
\end{proof}

\vspace{0.25cm}

As is the custom in the theory of Latin squares, let us call two binary operations $g_1,g_2\in\Omega_2(X)$ \emph{orthogonal} if $\{g_1,g_2\}$ is an orthogonal set. The following corollary follows straightforwardly from Theorem \ref{theorem_orthogonalQ}. In spite of its relevance, it seems not to have been noticed in the literature.

\begin{corollary}\label{corollary_orthogonalQ} 
Let $g_1,g_2\in\Omega_2(X)$ be orthogonal binary operations on $X$. Then there is a one-to-one correspondence between all quasigroup operations on $X$ and all binary operations on $X$ that are orthogonal to both $g_1$ and $g_2$.
\end{corollary}

\vspace{0.25cm}

Note that we do not claim that for any two orthogonal Latin squares $L_1$, $L_2$ the number of Latin squares that are orthogonal to both $L_1$ and $L_2$ is independent of these two Latin squares. Indeed, Proposition 4.3 yields all Latin squares orthogonal to both $L_1$ and $L_2$, but it also yields other non-Latin operations orthogonal to both $L_1$ and $L_2$.

\section{A pair of examples}

Let us finish our study with a pair of examples. The first one illustrates some of the results appearing in the previous sections, whereas the second one gives rise to new open questions to deal with.

\begin{example}\label{example1} Let $X=\{1,2\}$. The set $\Omega_2(X)$ is formed by the $16$ binary operations having the following Cayley tables. 

{\scriptsize \[\begin{array}{ccccccccc}
\begin{array}{c}\begin{array} {|c|c|} \hline
   1  &  1\\ \hline
   1  & 1 \\ \hline
\end{array}\\ g_1\end{array} & \begin{array}{c}\begin{array} {|c|c|} \hline
   1  &  1\\ \hline
   1  & 2 \\ \hline
\end{array}\\ g_2\end{array} & \begin{array}{c} \begin{array} {|c|c|} \hline
   1  &  1\\ \hline
   2  & 1 \\ \hline
\end{array}\\g_3\end{array}&  \begin{array}{c}\begin{array} {|c|c|} \hline
   1  &  1\\ \hline
   2  & 2 \\ \hline
\end{array}\\g_4\end{array} & \begin{array}{c}\begin{array} {|c|c|} \hline
   1  &  2\\ \hline
   1  & 1 \\ \hline
\end{array}\\ g_5\end{array} & \begin{array}{c}\begin{array} {|c|c|} \hline
   1  &  2\\ \hline
   1  & 2 \\ \hline
\end{array}\\ g_6\end{array} & \begin{array}{c}\begin{array} {|c|c|} \hline
   1  &  2\\ \hline
   2  & 1 \\ \hline
\end{array}\\ g_7\end{array} & \begin{array}{c}\begin{array} {|c|c|} \hline
   1  &  2\\ \hline
   2  & 2 \\ \hline
\end{array}\\ g_8\end{array}
\end{array}\]
\[\begin{array}{ccccccccc}
\begin{array}{c}\begin{array} {|c|c|} \hline
   2  &  1\\ \hline
   1  & 1 \\ \hline
\end{array}\\ g_9\end{array} & \begin{array}{c}\begin{array} {|c|c|} \hline
   2  &  1\\ \hline
   1  & 2 \\ \hline
\end{array}\\ g_{10}\end{array} & \begin{array}{c} \begin{array} {|c|c|} \hline
   2  &  1\\ \hline
   2  & 1 \\ \hline
\end{array}\\g_{11}\end{array}&  \begin{array}{c}\begin{array} {|c|c|} \hline
   2  &  1\\ \hline
   2  & 2 \\ \hline
\end{array}\\g_{12}\end{array} & \begin{array}{c}\begin{array} {|c|c|} \hline
   2  &  2\\ \hline
   1  & 1 \\ \hline
\end{array}\\ g_{13}\end{array} & \begin{array}{c}\begin{array} {|c|c|} \hline
   2  &  2\\ \hline
   1  & 2 \\ \hline
\end{array}\\ g_{14}\end{array} & \begin{array}{c}\begin{array} {|c|c|} \hline
   2  &  2\\ \hline
   2  & 1 \\ \hline
\end{array}\\ g_{15}\end{array} & \begin{array}{c}\begin{array} {|c|c|} \hline
   2  &  2\\ \hline
   2  & 2 \\ \hline
\end{array}\\ g_{16}\end{array}
\end{array}\]}

In particular, $\mathcal{Q}_2(X)=\{g_7,\,g_{10}\}$. The next array constitutes the Cayley table of the binary groupoid $(\Omega_2(X),\odot_{g_7})$.

{\small \[\begin{array}{|c|c|c|c|c|c|c|c|c|c|c|c|c|c|c|c|} \hline
g_1 & g_2 & g_3 & g_4 & g_5 & g_6 & g_7 & g_8 & g_9 & g_{10}& g_{11} & g_{12} & g_{13} & g_{14} & g_{15} & g_{16}\\ \hline 
g_2&g_1&g_4&g_3&g_6&g_5&g_8&g_7&g_{10}&g_9&g_{12}&g_{11}&g_{14}&g_{13}&g_{16}&g_{15}\\ \hline
g_3&g_4&g_1&g_2&g_7&g_8&g_5&g_6&g_{11}&g_{12}&g_9&g_{10}&g_{15}&g_{16}&g_{13}&g_{14}\\ \hline
g_4&g_3&g_2&g_1&g_8&g_7&g_6&g_5&g_{12}&g_{11}&g_{10}&g_9&g_{16}&g_{15}&g_{14}&g_{13}\\ \hline
g_5&g_6&g_7&g_8&g_1&g_2&g_3&g_4&g_{13}&g_{14}&g_{15}&g_{16}&g_9&g_{10}&g_{11}&g_{12}\\ \hline
g_6&g_5&g_8&g_7&g_2&g_1&g_4&g_3&g_{14}&g_{13}&g_{16}&g_{15}&g_{10}&g_9&g_{12}&g_{11}\\ \hline
g_7&g_8&g_5&g_6&g_3&g_4&g_1&g_2&g_{15}&g_{16}&g_{13}&g_{14}&g_{11}&g_{12}&g_9&g_{10}\\ \hline
g_8&g_7&g_6&g_5&g_4&g_3&g_2&g_1&g_{16}&g_{15}&g_{14}&g_{13}&g_{12}&g_{11}&g_{10}&g_9\\ \hline
g_9&g_{10}&g_{11}&g_{12}&g_{13}&g_{14}&g_{15}&g_{16}&g_1&g_2&g_3&g_4&g_5&g_6&g_7&g_8\\ \hline
g_{10}&g_9&g_{12}&g_{11}&g_{14}&g_{13}&g_{16}&g_{15}&g_2&g_1&g_4&g_3&g_6&g_5&g_8&g_7\\ \hline
g_{11}&g_{12}&g_9&g_{10}&g_{15}&g_{16}&g_{13}&g_{14}&g_3&g_4&g_1&g_2&g_7&g_8&g_5&g_6\\ \hline
g_{12}&g_{11}&g_{10}&g_9&g_{16}&g_{15}&g_{14}&g_{13}&g_4&g_3&g_2&g_1&g_8&g_7&g_6&g_5\\ \hline
g_{13}&g_{14}&g_{15}&g_{16}&g_9&g_{10}&g_{11}&g_{12}&g_5&g_6&g_7&g_8&g_1&g_2&g_3&g_4\\ \hline
g_{14}&g_{13}&g_{16}&g_{15}&g_{10}&g_9&g_{12}&g_{11}&g_6&g_5&g_8&g_7&g_2&g_1&g_4&g_3\\ \hline
g_{15}&g_{16}&g_{13}&g_{14}&g_{11}&g_{12}&g_9&g_{10}&g_7&g_8&g_5&g_6&g_3&g_4&g_1&g_2\\ \hline
g_{16}&g_{15}&g_{14}&g_{13}&g_{12}&g_{11}&g_{10}&g_9&g_8&g_7&g_6&g_5&g_4&g_3&g_2&g_1\\ \hline
\end{array}\]}

That is, $(X,g_7)$ is isomorphic to the abelian group $(\mathbb{Z}_2,+)$, whereas $(\Omega_2(X),\odot_{g_7})$ is isomorphic to the abelian group $(\mathbb{Z}_2\times \mathbb{Z}_2\times \mathbb{Z}_2\times \mathbb{Z}_2,+)$. That is, associativity and commutativity are preserved, as Lemma \ref{lemma_identities} indicates. Moreover, the identity element $1$ in $(X,g_7)$ maps to the identity element $g_1=\gamma_1$ in $(\Omega_2(X),\odot_{g_7})$, as Proposition \ref{proposition:LI} shows. By abuse of notation, we could write $\odot_{\mathbb{Z}_2}=\mathbb{Z}_2\times \mathbb{Z}_2\times \mathbb{Z}_2\times \mathbb{Z}_2$. 

\hfill $\lhd$
\end{example}

\vspace{0.25cm}

Note in the previous example that the groupoid $(X,g_7)$ is embedded into $(\Omega_2(X),$ $\odot_{g_7})$ many more times than the way in which Lemma \ref{lemma_embedding} indicates. The next example enables us to delve into this aspect by considering isomorphic embeddings, which gives rise in turn to the open question about how many embeddings actually exist from an $m$-ary groupoid $(X,f)$ to the $m$-ary groupoid $(\Omega_m(X),\odot)$.

\begin{example}\label{example2} Let $X=\{1,2,3,4,5\}$ and $\star,\circ,*\in \Omega_2(X)$, with respective Cayley tables
\[A_\star\equiv\begin{array}{|c|c|c|c|c|} \hline
1 & 5 & 4 & 3 & 2\\ \hline
3 & 2 & 1 & 5 & 4\\ \hline
5 & 4 & 3 & 2 & 1\\ \hline
2 & 1 & 5 & 4 & 3\\ \hline
4 & 3 & 2 & 1 & 5\\ \hline
\end{array} \hspace{1cm} A_\circ\equiv\begin{array}{|c|c|c|c|c|} \hline
1 & 2 & 3 & 4 & 5\\ \hline
1 & 2 & 3 & 4 & 5\\ \hline
1 & 2 & 3 & 4 & 5\\ \hline
1 & 2 & 3 & 4 & 5\\ \hline
1 & 2 & 3 & 4 & 5\\ \hline
\end{array} \hspace{1cm} A_*\equiv\begin{array}{|c|c|c|c|c|} \hline
1 & 4 & 2 & 5 & 3\\ \hline
4 & 2 & 5 & 3 & 1\\ \hline
2 & 5 & 3 & 1 & 4\\ \hline
5 & 3 & 1 & 4 & 2\\ \hline
3 & 1 & 4 & 2 & 5\\ \hline
\end{array}.
\]
In addition, let $\star^t,\circ^t \in \Omega_2(X)$ be the binary operations on $X$ whose Cayley tables are respectively the transposes of $A_\star$ and $A_\circ$. In particular, $\{\star,\,\star^t\}$ is a set of orthogonal binary operations. (That is, $f$ is {\em self-orthogonal}.) Then, the binary operation $\odot_\star$ applied to the set $\{\star,\star^t,\circ,\circ^t,*\}$ gives rise to the following Cayley table.

\[\begin{array}{c||c|c|c|c|c|c|}
\odot_\star & \star & \star^t & \circ & \circ^t & *\\ \hline \hline 
\star & \star & \circ & \star^t & * & \circ^t\\ \hline
\star^t & \circ^t & \star^t & * & \star & \circ\\ \hline
\circ & * & \circ^t & \circ & \star^t & \star\\ \hline
\circ^t & \circ & * & \star & \circ^t & \star^t\\ \hline
* & \star^t & \star & \circ^t & \circ & *\\ \hline
\end{array} \hspace{1cm} \equiv \hspace{1cm} \begin{array}{c||c|c|c|c|c|c|}
\bullet & 1 & 2 & 3 & 4 & 5\\ \hline \hline 
1 & 1 & 3 & 2 & 5 & 4\\ \hline
2 & 4 & 2 & 5 & 1 & 3\\ \hline
3 & 5 & 4 & 3 & 2 & 1\\ \hline
4 & 3 & 5 & 1 & 4 & 2\\ \hline
5 & 2 & 1 & 4 & 3 & 5\\ \hline
\end{array}
\]

\noindent Here, $(X,\bullet)$ is obtained after replacing, respectively, the symbols $\star$, $\star^t$, $\circ$, $\circ^t$ and $*$ by the symbols $1$, $2$, $3$, $4$ and $5$. Note that $(X,\star)$ and $(X,\bullet)$ are isomorphic by means of the isomorphism defined by the permutation $(235)\in\mathcal{S}_5$. \hfill $\lhd$
\end{example}

\vspace{0.25cm}

\subsection*{Acknowledgments}

Falc\'on's work is partially supported by  the project TED2021-130566B-100 from Ministry of Science and Innovation of the Government of Spain. Mella’s work is partially supported by INdAM - GNSAGA, and by the University of Modena and Reggio Emilia (FAR2023).
The research was also supported by the Italian Ministry of University and Research through the project named “Fondi per attivit\`a a carattere internazionale” 2022 INTER DICATAM PASOTTI.  P. Vojt\v{e}chovsk\'{y} supported by the Simons Foundation Mathematics and Physical Sciences Collaboration Grant for Mathematicians no. 855097, and by the Enhanced Sabbatical Program of the University of Denver.

\end{document}